\documentclass{article}
\usepackage{amsfonts,amsthm}
\usepackage{authblk}
\usepackage{newpxtext,newpxmath}
\usepackage{pdflscape}
\usepackage{multirow}
\usepackage{multicol}
\usepackage{tikz}
\usetikzlibrary{arrows}
\usepackage{geometry}

\usepackage{arydshln}
\usepackage{mathrsfs}

\theoremstyle{plain}
\newtheorem{theorem}{Theorem}[section]

\newtheorem{proposition}[theorem]{Proposition}
\newtheorem{corollary}[theorem]{Corollary}

\theoremstyle{definition}
\newtheorem{definition}{Definition}[section]
\newtheorem{notation}{Notation}[section]

\newtheorem{example}{Example}[section]

\theoremstyle{remark}
\newtheorem{remark}{Remark}[section]

\title{Extended derivations of algebras}

\makeatletter
\def\blfootnote{\xdef\@thefnmark{}\@footnotetext}
\makeatother

\makeatletter
\renewcommand\@date
{{%

\vspace{-1.5cm}%
\large\centering

  \begin{tabular}{@{}c@{}}
    Edison Alberto \textsc{Fern\'andez-Culma} \textsuperscript{$\dagger$}\blfootnote{${}^{\dagger}$ Partially supported by CONICET and SECyT-UNC: 33620180100523CB} \\
{\small Centro de Investigaci\'on y Estudios en Matem\'atica de C\'ordoba, CONICET}\\
{\small Facultad de Matem\'atica, Astronom\'ia, F\'isica y Computaci\'on, UNC}\\
{\small C\'ordoba, Argentina}\\
    \texttt{\small efernandez@famaf.unc.edu.ar}
  \end{tabular}%
}}
\makeatother

\author{}

\begin{document}

\maketitle

\blfootnote{2010 \textit{Mathematics Subject Classification:} 16W25. \textit{Key words and phrases:} Lie algebras; Anti-commutative algebras; Generalized derivations of Algebras; Isomorphism problem; Invariants of algebras; Degenerations of Algebras }

\begin{abstract}
We study the concept of \textit{extended derivations of algebras} which expands diverse definitions of \textit{generalized derivations} given in the literature. We concentrate on the family of the anti-commutative algebras and classify such spaces of derivations by considering a suitable notion of equivalence when we restrict attention to such algebras. Afterwards, we investigate a link between the well-known $(\alpha,\beta,\gamma)$-derivations and some new extended derivations obtained in the above-mentioned classification. Finally, we present applications of extended derivations to study degenerations of algebras.
\end{abstract}

\section{Introduction}{$ $}

As far as we know, the notion of derivations of algebras was motivated by Nathan Jacobson in \cite{jacobson}; in analogy with the \textit{Leibniz product rule} of calculus (although at that time there were evidences of a link between derivations and automorphisms of algebras, see \cite[p.206-fn.\ddag ]{jacobson}).  About three decades later, Nicolas Bourbaki published a new edition of his treatise on Algebra for his \textit{{\'{E}l\'{e}ments de math\'{e}matique}} (formerly known as \textit{Les structures fondamentales de l'analyse}). In \cite[Chapitre III. \S10.]{Bourbaki1}, it is given the \textit{general definition of derivations} and properties of such notion are deeply studied throughout this section. Since then, derivations of algebras and their ``derived products'' have been extensively studied with the aim of providing structure results for algebras. For instance, we can mention the Jacobson-Moens theorem, which states that

\begin{quote}
\textbf{Theorem} \cite{jacobson2,moens} A Lie algebra over a field of characteristic zero is nilpotent if and only if it admits an invertible \textit{Leibniz-derivation}.
\end{quote}

Also, it is worth mentioning the works by Leger and Luks \cite{legerluks}, and by Hrivn\'{a}k and Novotn\'{y} \cite{NovotnyHrivnak}. In the first one, some relevant sufficient criteria are given for ensuring the equality between the sum of the algebra of derivations and the \textit{centroid} of a Lie algebra and its algebra of \textit{quasiderivations}, as well as conditions on a Lie algebra that force equality in the inclusion of its centroid into its \textit{quasicentroid}; among other remarkable results. In regards to the second paper, the authors present the concept of \textit{$(\alpha, \beta, \gamma)$-derivations} of Lie algebras. Such idea was introduced by them in their respective Ph.D. theses for the purpose of defining new invariants of Lie algebra, invariants that can be used to study \textit{degenerations} of Lie algebras (see \cite[Chapter 2]{hrivnak} and \cite[Section 4.7]{Novotny}). A very nice result concerning $(\alpha, \beta, \gamma)$-derivations can be interpreted as follows:

\begin{quote}
\textbf{Theorem} \cite[Theorem 3.4]{NovotnyHrivnak2} There are sufficient invariants defined by $(\alpha, \beta, \gamma)$-derivations to distinguish $3$-dimensional complex Lie algebras, up to isomorphism.
\end{quote}

The aim of this paper is to studied \textit{extended derivations} of algebras, which can be considered as a generalization of the different notions of derivations mentioned above. We have used extended derivations to classify algebras endowed with additional structures coming from geometry or algebra, up to an appropriate equivalence relation \cite{fernandezrojas1,fernandezrojas2,fernandezrojas3}. We will restrict ourselves to study extended derivations in the family of anti-commutative algebras (results also hold for commutative algebras) and we introduce a notion of equivalence of extended derivations. In \textbf{Theorem \ref{teoremaA}}, we classify extended derivations, up to the mentioned equivalence relation and recover \cite[Theorem 2.2]{NovotnyHrivnak}. Then we studied two types of extended derivations obtained in the classification, denoted by $\mathcal{T}(t)(\square)$ and $\mathcal{P}(\square)$; the first one depends on a parameter $t$ in the base field. We are interested in the possibility of expressing them in terms of $(\alpha, \beta, \gamma)$-derivations and we obtain some results about that.

The proofs in this paper are quite elementary and follow from basic linear algebra. In fact, one advantage of extended derivations is that they define invariants of algebras which can be easily computed by solving a homogeneous system of linear equations. Besides such invariants provide tools to study \textit{degenerations of algebras}; we present an illustrative example in Section \ref{applicaciones}.

\section{Extended derivations}

\begin{definition}
Let $\mathbb{K}$ be a field. An \textit{algebra over} $\mathbb{K}$ is simply a vector space over $\mathbb{K}$, say $V$, endowed with a bilinear map $\mu:V \times V \rightarrow V$ called \textit{product}. An algebra over $\mathbb{K}$ is sometimes also called a $\mathbb{K}$-algebra, and $\mathbb{K}$ is called the \textit{base field}.

A $\mathbb{K}$-algebra $\mathfrak{A}=(V,\mu)$ is said to be \textit{anti-commutative} (or \textit{skew-symmetric}) if it satisfies the identity
  $$
  \mu(X,Y) = - \mu(Y,X)  \quad \mbox{ for all } X,Y\in V.
  $$
\end{definition}

The \textit{Lie algebras} are widely considered to be the most remarkable example of anti-commutative algebras.

For simplicity, we assume throughout this paper that all $\mathbb{K}$-algebras are finite dimensional and that the base field $\mathbb{K}$ is of characteristic zero.

\begin{notation}
If $V$ and $W$ are finite-dimensional $\mathbb{K}$-vector spaces, let us denote the vector space of all multilinear maps from $V^{k}$, the $k$-fold Cartesian product of $V$, to $W$ by $L^{k}(V;W)$.
Let $C^{k}(V;W)$ denote the subspace consisting of all multilinear maps in $ L^{k}(V;W)$ that are \textit{alternating}; i.e. $\mu \in L^{k}(V;W)$ such that $\mu(v_{\sigma(1)}, v_{\sigma(2)}, \ldots , v_{\sigma(k)})$ is equal to $\operatorname{sign}(\sigma) \, \mu(v_1,v_2, \ldots, v_k)$ for each permutation $\sigma \in \operatorname{S}_{k}$.
\end{notation}

\begin{definition}
Let $\mathfrak{A}=(V,\mu)$ and $\mathfrak{B}=(W,\lambda)$ be two $\mathbb{K}$-algebras. An invertible linear transformation $T:V \rightarrow W$ is called an \textit{algebra isomorphism} if it satisfies
$$
T(\mu(X,Y)) = \lambda (T(X),T(Y))  \quad \mbox{ for all } X,Y\in V.
$$
The two algebras are said to be \textit{isomorphic} if there exists an algebra isomorphism between $\mathfrak{A}$ and $\mathfrak{B}$.
\end{definition}

Next we introduce the focal point of our study.

\begin{definition}\label{defextended}
Let $\mathfrak{A}=(V,\mu)$ be an $\mathbb{K}$-algebra and let $\Omega=\left(\begin{array}{ccc} a_{1} & a_{2}& a_{3} \\ a_{4} & a_{5} & a_{6} \\ \end{array} \right)$ be any fixed matrix in $\mathrm{M}(2\times 3 , \mathbb{K})$. An \textit{$\Omega-\mbox{derivation}$ of $\mathfrak{A}$} is a triple of linear transformations $(A,B,C) \in (L^{1}(V;V))^{3}$ such that
$$
  \left\{
  \begin{array}{l}
    a_1 A + a_2 B + a_3 C =0\\
    a_4 A + a_5 B + a_6 C =0\\
    A\mu(X,Y) = \mu(BX,Y) + \mu(X,CY) \quad \mbox{ for all } X,Y\in V.
  \end{array}
  \right.
$$
We denote by $\operatorname{Der}_{\Omega}(\mathfrak{A})$ the set of all $\Omega-\mbox{derivations}$ of $\mathfrak{A}$. We say that the vector space $\operatorname{Der}_{\Omega}(\mathfrak{A})$ is a space of \textit{extended derivations}.
\end{definition}

\begin{example}
Let $\alpha_{1}, \alpha_{2}$ and $\alpha_{3}$ three constants in $\mathbb{K}$ and consider the $2\times 3$ matrices $\Omega_{1}$, $\Omega_{2}$ and $\Omega_{3}$ in $\mathrm{M}(2\times 3 , \mathbb{K})$ given by
$$
\begin{array}{ccccc}
 \Omega_{1}=\left(\begin{array}{ccc} \alpha_{2} & -\alpha_{1} & 0 \\ \alpha_{3} & 0 & -\alpha_{1} \\ \end{array} \right),  &
    &
 \Omega_{2}=\left(\begin{array}{ccc} \alpha_{2} & -\alpha_{1} & 0 \\     0  & \alpha_{3} & -\alpha_{2} \\ \end{array} \right),   &
    &
 \Omega_{3}=\left(\begin{array}{ccc} \alpha_{3} & 0 & -\alpha_{1} \\ 0 & \alpha_{3} & -\alpha_{2} \\ \end{array} \right).
\end{array}
$$

Note that if $\alpha_{i}\neq 0$, then the set of all $(\alpha_1,\alpha_2,\alpha_3)-\mbox{derivations}$ of an algebra $\mathfrak{A}$,  $\mathcal{D}(\alpha_1,\alpha_2,\alpha_3)(\mathfrak{A})$ (see \cite[\S 2]{NovotnyHrivnak}), is isomorphic to $\operatorname{Der}_{\Omega_{i}}(\mathfrak{A})$. In particular, the algebra of derivations of $\mathfrak{A}$, $\operatorname{Der}(\mathfrak{A})$, is isomorphic to $\operatorname{Der}_{\Omega_{1}}(\mathfrak{A})$, with $\alpha_{1}=\alpha_{2}=\alpha_{3}$ and $\alpha_{1}\neq0$.
\end{example}

\begin{remark}
Two different matrices $\Omega_1$ and $\Omega_2$ in $\mathrm{M}(2\times 3 , \mathbb{K})$ could produce the same space derivations; in the sense of Definition \ref{defextended}. In fact, if $\Omega_1$ and $\Omega_2$ are \textit{row equivalent matrices}, then $\operatorname{Der}_{\Omega_{1}}(\mathfrak{A}) = \operatorname{Der}_{\Omega_{2}}(\mathfrak{A})$.
\end{remark}

If an algebra $\mathfrak{A}$ satisfies some \textit{symmetry condition} such as anti-commutativity (or commutativity), then there are more redundancies between the different spaces of extended derivations in the following sense:

\begin{definition}
Let $\Omega$ and $\widetilde{\Omega}$ be two matrices in $\mathrm{M}(2\times 3 , \mathbb{K})$. We say that the two spaces of extended derivations associated with them are \textit{equivalent with respect to the family of anti-commutative algebras} if for each anti-commutative algebra $\mathfrak{A}$ there exists a linear isomorphism $\varphi$ between $\operatorname{Der}_{\Omega}(\mathfrak{A})$ and $\operatorname{Der}_{\widetilde{\Omega}}(\mathfrak{A})$.
\end{definition}

\begin{example}
Let $\Omega$ be a matrix in $\mathrm{M}(2\times 3 , \mathbb{K})$, and let $\widetilde{\Omega}$ the matrix in $\mathrm{M}(2\times 3 , \mathbb{K})$ obtained
by interchanging columns 2 and 3. If $\mathfrak{A}=(V,\mu)$ is any anti-commutative algebra, then it is not hard to check that the linear transformation $\varphi: \operatorname{Der}_{\Omega}(\mathfrak{A}) \rightarrow \operatorname{Der}_{\widetilde{\Omega}}(\mathfrak{A}) $ given by $\varphi(A,B,C) = (A,C,B)$ is an isomorphism. Then $\operatorname{Der}_{\Omega}( \square )$ is equivalent to $\operatorname{Der}_{\widetilde{\Omega}}( \square )$ in such family of algebras.
\end{example}

Our purpose in this section is to classify, up to equivalence, all possible spaces of extended derivations of anti-commutative algebras. To do this, we use the following result which says that $\operatorname{Der}_{0_{2\times 3}}(\mathfrak{A})$ ($\mbox{\LARGE $\vartriangle$}(\mathfrak{A})$ in the notation used in \cite{legerluks}) can be decomposed as a direct product of simpler spaces of extended derivations.

Given an algebra $\mathfrak{A}=(V,\mu)$, let us define the space of \textit{nearly derivations} of $\mathfrak{A}$, denoted by $\operatorname{NDer}(\mathfrak{A})$, to be the set
$$
\operatorname{NDer}(\mathfrak{A}) := \{ (P,Q) \in L^{1}(V;V)\times L^{1}(V;V) :  P\mu(X,Y) = \mu(QX,Y) + \mu(X,QY) \mbox{ for all } X,Y \in V \}.
$$
If $(P,Q)$ is a nearly derivation of $\mathfrak{A}$ then $Q$ is a \textit{quasiderivation} of $\mathfrak{A}$, in the sense of \cite[\S 1]{legerluks}.  The vector space $\operatorname{NDer}(\mathfrak{A})$ is isomorphic to
$\operatorname{Der}_{{\Omega}}(\mathfrak{A})$ with $\Omega=\left(\begin{array}{ccc} 0 & 1 & -1  \\ 0 & 0 & 0 \\ \end{array} \right)$. Note that $\operatorname{NDer}(\mathfrak{A})$ is the Lie algebra of the isotropy group of $\mu$ for the action of $\operatorname{GL}(V)\times \operatorname{GL}(V)$ on $L^{2}(V;V)$ given by $((g,h)\bullet \lambda)(X,Y):=g \lambda(h^{-1}X,h^{-1}Y)$.

The \textit{quasicentroid of $\mathfrak{A}$} (in the sense of \cite[\S 1]{legerluks}), denoted by $\operatorname{QC}(\mathfrak{A})$, is the set
$$
\operatorname{QC}(\mathfrak{A}):= \{ R \in L^{1}(V;V) : \mu(R X, Y) - \mu(X, R Y) =0 \mbox{ for all } X,Y \in V  \}.
$$
We have $\operatorname{QC}(\mathfrak{A})$  is equal to $\mathcal{D}(0,1,-1)(\mathfrak{A})$ and isomorphic to $\operatorname{Der}_{{\Omega}}(\mathfrak{A})$ with $\Omega=\left(\begin{array}{ccc} 0 & 1 & 1  \\ 1 & 0 & 0 \\ \end{array} \right)$. In \cite[Pag. 211]{NovotnyHrivnak} it was noted that $\mathcal{D}(0,1,-1)(\mathfrak{A})$ endowed with the usual \textit{Jordan product} on $L^{1}(V;V)$ is a \textit{Jordan algebra}.

\begin{proposition}{\cite[Proof of Proposition 3.3]{legerluks}}\label{descomposicion}
Let  $\mathfrak{A}=(V,\mu)$ be an anti-commutative (or commutative) algebra. Then $\operatorname{Der}_{0_{2\times 3}}(\mathfrak{A})$ is isomorphic to $\operatorname{NDer}(\mathfrak{A}) \times \operatorname{QC}(\mathfrak{A})$
\end{proposition}

\begin{proof}
  Let $(A,B,C) \in \operatorname{Der}_{0_{2\times 3}}(\mathfrak{A})$ and let $X, Y \in V$ be arbitrary. By definition of $\operatorname{Der}_{0_{2\times 3}}(\mathfrak{A})$:
\begin{eqnarray}
\label{simetria0}  A\mu(X,Y) &=& \mu(BX,Y) + \mu(X,CY)\\
\label{simetria}   A\mu(Y,X) &=& \mu(BY,X) + \mu(Y,CX).
\end{eqnarray}
It follows from anti-commutativity (or commutativity) of $\mathfrak{A}$ and Equation (\ref{simetria}) that
\begin{eqnarray}
\label{simetria2}  A\mu(X,Y) &=& \mu(C X,Y) + \mu(X,BY).
\end{eqnarray}
By averaging Equations (\ref{simetria0}) and (\ref{simetria2}), and  since $\mu$ is a bilinear map we get
\begin{eqnarray*}
   A\mu(X,Y) & = & \mu \left( \frac{1}{2}(B+C) X , Y \right) +  \mu \left( X , \frac{1}{2}(B+C) Y \right).
\end{eqnarray*}
Now, subtracting Equation (\ref{simetria2}) from Equation (\ref{simetria0}), and by bilinearity again, we obtain
\begin{eqnarray*}
  0 & = & \mu \left( \frac{1}{2}(B - C) X , Y \right) -  \mu \left( X , \frac{1}{2}(B - C) Y \right).
\end{eqnarray*}
The last two equations motivates the following isomorphism $\varphi : \operatorname{Der}_{0_{2\times 3}}(\mathfrak{A}) \rightarrow \operatorname{NDer}(\mathfrak{A}) \times \operatorname{QC}(\mathfrak{A})$:
\begin{eqnarray*}
\varphi(A,B,C) & = & \left( \left( A, \frac{1}{2}(B+C)\right),  \frac{1}{2}(B-C) \right),
\end{eqnarray*}
whose inverse is the linear transformation $\varphi \operatorname{NDer}(\mathfrak{A}) \times \operatorname{QC}(\mathfrak{A}) \rightarrow \operatorname{Der}_{0_{2\times 3}}(\mathfrak{A})$ defined by the formula
\begin{eqnarray*}
\varphi^{-1}( (P,Q) , R) & = & \left( P, Q+R , Q-R \right).
\end{eqnarray*}
\end{proof}

\begin{notation}
Let  $\mathfrak{A}=(V,\mu)$ be an anti-commutative (or commutative) algebra and let $\Omega=\left(\begin{array}{ccc} b_{1} & b_{2}& b_{3} \\ b_{4} & b_{5} & b_{6} \\ \end{array} \right)$ in $\mathrm{M}(2\times 3 , \mathbb{K})$.
We denote by $(\operatorname{NDer}(\mathfrak{A}) \times \operatorname{QC}(\mathfrak{A}))\Omega$ the vector space
$$
\left\{ ( (P,Q) , R) \in \operatorname{NDer}(\mathfrak{A}) \times \operatorname{QC}(\mathfrak{A}) :
\begin{array}{l}
b_{1} P + b_{2} Q + b_{3}R = 0 \\
b_{4} P + b_{5} Q + b_{6}R = 0 \\
\end{array}
\right\}
$$
\end{notation}

\begin{corollary}\label{isoextended}
Let  $\mathfrak{A}=(V,\mu)$ be an anti-commutative (or commutative) algebra and let $\Omega=\left(\begin{array}{ccc} a_{1} & a_{2}& a_{3} \\ a_{4} & a_{5} & a_{6} \\ \end{array} \right)$
and  $\mho=\left(\begin{array}{ccc} b_{1} & b_{2}& b_{3} \\ b_{4} & b_{5} & b_{6} \\ \end{array} \right)$ in $\mathrm{M}(2\times 3 , \mathbb{K})$. Then $\operatorname{Der}_{ \Omega }(\mathfrak{A})$ is isomorphic to $(\operatorname{NDer}(\mathfrak{A}) \times \operatorname{QC}(\mathfrak{A}))\widetilde{\Omega}$ where $\widetilde{\Omega}$ is given by
$$
\left(\begin{array}{ccc} a_{1} &a_2+a_3& a_2-a_3  \\ a_{4} &  a_5+a_6  &  a_5-a_6  \\ \end{array} \right)
$$
and $(\operatorname{NDer}(\mathfrak{A}) \times \operatorname{QC}(\mathfrak{A})) \mho $ is isomorphic to $\operatorname{Der}_{ \widetilde{\mho }}(\mathfrak{A})$ with $\widetilde{\mho }$ given by
$$
\left(\begin{array}{ccc} b_{1} &(b_2+b_3)/2& (b_2-b_3)/2  \\ b_{4} &  (b_5+b_6)/2  &  (b_5-b_6)/2  \\ \end{array} \right)
$$
\end{corollary}

\begin{proof}
Let $\varphi : \operatorname{Der}_{0_{2\times 3}}(\mathfrak{A}) \rightarrow \operatorname{NDer}(\mathfrak{A}) \times \operatorname{QC}(\mathfrak{A})$ be the isomorphism used in the preceding proof.
The restriction of $\varphi$ to the vector subspace $\operatorname{Der}_{ \Omega }(\mathfrak{A})$ yields an isomorphism between $\operatorname{Der}_{ \Omega }(\mathfrak{A})$ and $(\operatorname{NDer}(\mathfrak{A}) \times \operatorname{QC}(\mathfrak{A}))\widetilde{\Omega}$.

Similarly, the restriction of the isomorphism $\varphi^{-1}$ to $(\operatorname{NDer}(\mathfrak{A}) \times \operatorname{QC}(\mathfrak{A})) \mho $ gives us an isomorphism to $\operatorname{Der}_{ \widetilde{\mho }}(\mathfrak{A})$.
\end{proof}

The advantage of using the spaces $(\operatorname{NDer}(\mathfrak{A}) \times \operatorname{QC}(\mathfrak{A})) \mho $ instead of the spaces $\operatorname{Der}_{ \Omega }(\mathfrak{A})$ lies in the fact that although the quasicentroid component of a triple $((P,Q),R)$ in $(\operatorname{NDer}(\mathfrak{A}) \times \operatorname{QC}(\mathfrak{A}))\mho $ could depend linearly on the other components, the \textit{quasicentroid equation} depends only of $R$.

\begin{theorem}\label{teoremaA}
  Let  $\mathfrak{A}=(V,\mu)$ be an anti-commutative (or commutative) algebra and let $\Omega \in \mathrm{M}(2\times 3 , \mathbb{K})$. Then $\operatorname{Der}_{ \Omega }(\mathfrak{A})$ isomorphic to one of the following spaces of extended derivations:
  \begin{itemize}
  \begin{multicols}{2}
    \item $\mathcal{D}(1,0,0)(\mathfrak{A})$
    \item $\mathcal{D}(0,1,-1)(\mathfrak{A}) = \operatorname{QC}(\mathfrak{A})$
    \item $\mathcal{D}(1,1,-1)(\mathfrak{A})$
    \item $\mathcal{D}(t,1,0)(\mathfrak{A})$, $t\in \mathbb{K}$
    \item $\mathcal{D}(t,1,1)(\mathfrak{A})$, $t\in \mathbb{K}$
    \item $\mathcal{D}(t,1,1)(\mathfrak{A}) \times \mathcal{D}(0,1,-1)(\mathfrak{A})$, $t\in \mathbb{K}$
    \item $\mathcal{D}(1,0,0)(\mathfrak{A}) \times \mathcal{D}(0,1,-1)(\mathfrak{A})$
  \end{multicols}
  \vspace{-0.5cm}
    \item $\mathcal{T}(t)(\mathfrak{A}) := \left\{ (A,B,C) \in \operatorname{Der}_{0_{2\times 3}}(\mathfrak{A}) : A + t B + (t-1)C=0 \right\}$, $t\in \mathbb{K}$
    \item $\mathcal{P}(\mathfrak{A}):= \left\{ (A,B) \in L^{1}(V;V)\times L^{1}(V;V) :  A\mu(X,Y) = \mu(BX,Y) \mbox{ for all } X,Y \in V \right\}$.
    \item $\operatorname{NDer}(\mathfrak{A})$
    \item $\operatorname{NDer}(\mathfrak{A}) \times \mathcal{D}(0,1,-1)(\mathfrak{A})$
  \end{itemize}
\end{theorem}

\begin{proof}
If $\Omega_1$ and $\Omega_2$ are row equivalent matrices, then $(\operatorname{NDer}(\mathfrak{A}) \times \operatorname{QC}(\mathfrak{A})) \Omega_1$ is equal to $(\operatorname{NDer}(\mathfrak{A}) \times \operatorname{QC}(\mathfrak{A})) \Omega_2$. Since there are only seven types of \textit{reduced row echelon forms} in $\mathrm{M}(2\times 3 , \mathbb{K})$, we have have to study $(\operatorname{NDer}(\mathfrak{A}) \times \operatorname{QC}(\mathfrak{A})) \Omega$ in each of the following cases:

\begin{itemize}
  \item Case I: $\Omega = \left(\begin{array}{ccc} 1 &0& a \\ 0 &  1  &  b  \\ \end{array} \right) $
  \begin{itemize}
    \item Subcase 1: $a=0$ and $b=0$. In this situation, we have $P=0$ and $Q=0$. Therefore $(\operatorname{NDer}(\mathfrak{A}) \times \operatorname{QC}(\mathfrak{A})) \Omega$ is naturally isomorphic to $\operatorname{QC}(\mathfrak{A})$.
    \item Subcase 2: $a\neq 0$ and $b=0$. Consider $\widetilde{\Omega} \in \mathrm{M}(2\times 3 , \mathbb{K})$ given by  $\left(\begin{array}{ccc} 1 &0& 1 \\ 0 &  1  & 0  \\ \end{array} \right)$ and let $\psi: (\operatorname{NDer}(\mathfrak{A}) \times \operatorname{QC}(\mathfrak{A})) \Omega \rightarrow (\operatorname{NDer}(\mathfrak{A}) \times \operatorname{QC}(\mathfrak{A})) \widetilde{\Omega}$ the linear map defined by $\psi ((P,Q),R) = ((P,Q),aR)$. Clearly $\psi$ is an isomorphism, and from Corollary \ref{isoextended} we have $(\operatorname{NDer}(\mathfrak{A}) \times \operatorname{QC}(\mathfrak{A})) \widetilde{\Omega}$ is isomorphic to $\operatorname{Der}_{\mho}(\mathfrak{A})$ with $\mho = \left(\begin{array}{ccc} 1 & 1/2 & -1/2 \\ 0 &  1/2  & 1/2  \\ \end{array} \right)$. Since $\mho$ is row equivalent to $\widetilde{\mho } =\left(\begin{array}{ccc} 1 & 0 & -1 \\ 0 &  1  & 1  \\ \end{array} \right)$ and $\operatorname{Der}_{\widetilde{\mho}}(\mathfrak{A})$ is isomorphic to $\mathcal{D}(1,1,-1)(\mathfrak{A})$, this implies that $(\operatorname{NDer}(\mathfrak{A}) \times \operatorname{QC}(\mathfrak{A})) \Omega$ is isomorphic to the last vector space.

        It is important to observe that $\mathcal{D}(1,1,-1)(\mathfrak{A}) = \operatorname{QC}(\mathfrak{A}) \cap \mathcal{D}(1,0,0)(\mathfrak{A})$.

    \item Subcase 3: $b\neq 0$. In this case, we have $(\operatorname{NDer}(\mathfrak{A}) \times \operatorname{QC}(\mathfrak{A})) \Omega$ is isomorphic to the vector space $\mathcal{D}(a/(2b),1,0)(\mathfrak{A})$. In fact, let $\widetilde{\Omega}$ be the matrix $\left(\begin{array}{ccc} 1 & 0 & a/b \\ 0 &  1  & 1  \\ \end{array} \right) $ and define the linear map $\psi: (\operatorname{NDer}(\mathfrak{A}) \times \operatorname{QC}(\mathfrak{A})) \Omega \rightarrow (\operatorname{NDer}(\mathfrak{A}) \times \operatorname{QC}(\mathfrak{A})) \widetilde{\Omega}$ by $\psi((P,Q),R) = ((P,Q),bR)$. It is easy to check that $\psi$ is an isomorphism, and since $(\operatorname{NDer}(\mathfrak{A}) \times \operatorname{QC}(\mathfrak{A})) \widetilde{\Omega}$ is isomorphic to $\mathcal{D}(a/(2b),0,1)(\mathfrak{A}) = \mathcal{D}(a/(2b),1,0)(\mathfrak{A})$ (by Corollary \ref{isoextended} again), the assertion follows.
  \end{itemize}
 \item  Case II: $\Omega = \left(\begin{array}{ccc} 1 & a & 0 \\ 0 &  0  &  1  \\ \end{array} \right) $. It is easily seen that $(\operatorname{NDer}(\mathfrak{A}) \times \operatorname{QC}(\mathfrak{A})) \Omega$ is clarly isomorphic to $\mathcal{D}(-a,1,1)(\mathfrak{A})$.

 \item Case III: $\Omega = \left(\begin{array}{ccc} 1 & a & b \\ 0 &  0  &  0  \\ \end{array} \right) $
  \begin{itemize}
    \item Subcase 1: $b=0$. In this case it is straightforward to see that $(\operatorname{NDer}(\mathfrak{A}) \times \operatorname{QC}(\mathfrak{A})) \Omega$ is isomorphic to $\mathcal{D}(-a,1,1)(\mathfrak{A}) \times \mathcal{D}(0,1,-1)(\mathfrak{A})$.
    \item Subcase 2: $b\neq 0$. Let $\widetilde{\Omega}$ be the matrix $\left(\begin{array}{ccc} 1 & a & 1 \\ 0 &  0  &  0  \\ \end{array} \right) $, and let $\psi: (\operatorname{NDer}(\mathfrak{A}) \times \operatorname{QC}(\mathfrak{A})) \Omega \rightarrow (\operatorname{NDer}(\mathfrak{A}) \times \operatorname{QC}(\mathfrak{A})) \widetilde{\Omega}$ be defined by $\psi ((P,Q),R) = ((P,Q),bR)$. We have $\psi$ is a isomorphism, and it then follows from Corollary \ref{isoextended} that $(\operatorname{NDer}(\mathfrak{A}) \times \operatorname{QC}(\mathfrak{A})) \Omega$ is isomorphic to
        $$
        \left\{ (A,B,C) \in \operatorname{Der}_{0_{2\times 3}}(\mathfrak{A}) : A + t B + (t-1)C=0 \right\},
        $$
        with $t = (a+1)/2$.
  \end{itemize}
  \item Case IV: $\Omega = \left(\begin{array}{ccc} 0 & 1 & 0 \\ 0 &  0  &  1  \\ \end{array} \right) $. It is immediate that $(\operatorname{NDer}(\mathfrak{A}) \times \operatorname{QC}(\mathfrak{A})) \Omega$ is isomorphic to $\mathcal{D}(1,0,0)(\mathfrak{A})$
  \item Case V:  $\Omega = \left(\begin{array}{ccc} 0 & 1 & a \\ 0 &  0  &  0  \\ \end{array} \right) $.
  \begin{itemize}
    \item Subcase 1: $a=0$. In this case we have $(\operatorname{NDer}(\mathfrak{A}) \times \operatorname{QC}(\mathfrak{A})) \Omega$ is isomorphic to $\mathcal{D}(1,0,0)(\mathfrak{A}) \times \mathcal{D}(0,1,-1)(\mathfrak{A}) $.
    \item Subcase 2: $a\neq0$. If $\widetilde{\Omega} = \left(\begin{array}{ccc} 0 & 1 & 1 \\ 0 &  0  &  0  \\ \end{array} \right)$, we have $\psi: (\operatorname{NDer}(\mathfrak{A}) \times \operatorname{QC}(\mathfrak{A})) \Omega \rightarrow (\operatorname{NDer}(\mathfrak{A}) \times \operatorname{QC}(\mathfrak{A})) \widetilde{\Omega}$ given by $\psi((P,Q),R)=((P,Q),a R)$ is an isomorphism, and since $(\operatorname{NDer}(\mathfrak{A}) \times \operatorname{QC}(\mathfrak{A})) \widetilde{\Omega}$ is isomorphic to $\operatorname{Der}_{ \mho }(\mathfrak{A})$ with  and $\mho = \left(\begin{array}{ccc} 0 & 1 & 0 \\ 0 &  0  &  0  \\ \end{array} \right)$, it follows that $(\operatorname{NDer}(\mathfrak{A}) \times \operatorname{QC}(\mathfrak{A})) \Omega$ is isomorphic to
        $$
        \left\{ (A,B) \in L^{1}(V;V)\times L^{1}(V;V) :  A\mu(X,Y) = \mu(BX,Y) \mbox{ for all } X,Y \in V \right\}.
        $$
  \end{itemize}
  \item Case VI: $\Omega = \left(\begin{array}{ccc} 0 & 0 & 1 \\ 0 &  0  &  0  \\ \end{array} \right) $. In this case $(\operatorname{NDer}(\mathfrak{A}) \times \operatorname{QC}(\mathfrak{A})) \Omega$ is isomorphic to $\operatorname{NDer}(\mathfrak{A})$.
  \item Case VII: $\Omega = 0_{2\times 3}$. In this last case $(\operatorname{NDer}(\mathfrak{A}) \times \operatorname{QC}(\mathfrak{A}))0_{2\times 3} \cong \operatorname{Der}_{0_{2\times 3}}(\mathfrak{A}) \cong \operatorname{NDer}(\mathfrak{A}) \times \operatorname{QC}(\mathfrak{A})$.
\end{itemize}

\end{proof}

In the following diagram, we summarize some canonical embeddings between the spaces of extended derivations obtained in Theorem \ref{teoremaA}:\newline$ $

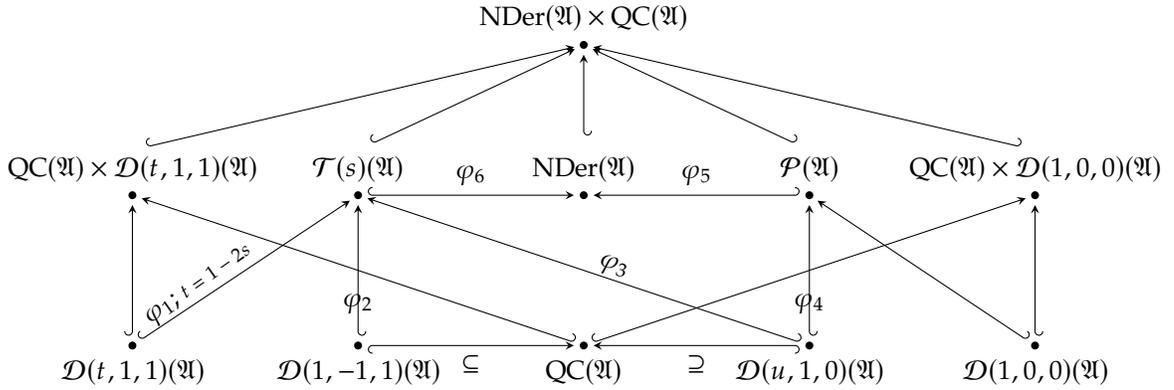
\begin{tikzpicture}

\node[circle,fill,inner sep=0pt,minimum size=3pt,label=above:{$\operatorname{NDer}(\mathfrak{A}) \times \operatorname{QC}(\mathfrak{A})$}] (nxq) at (0,4) {};

\node[](qdt11f) at (-6,2.6) {};
\node[](abctf) at (-3,2.6) {};
\node[](nderf) at (-0,2.6) {};
\node[](abf) at (3,2.6) {};
\node[](qd100f) at (6,2.6) {};

\node[circle,fill,inner sep=0pt,minimum size=3pt,label=above:{$\operatorname{QC}(\mathfrak{A}) \times \mathcal{D}(t,1,1)(\mathfrak{A}) $}] (qdt11) at (-6,2) {};
\node[circle,fill,inner sep=0pt,minimum size=3pt,label=above:{$\mathcal{T}(s)(\mathfrak{A})$}] (abct) at (-3,2) {};
\node[circle,fill,inner sep=0pt,minimum size=3pt,label=above:{$\operatorname{NDer}(\mathfrak{A})$}] (nder) at (0,2) {};
\node[circle,fill,inner sep=0pt,minimum size=3pt,label=above:{$\mathcal{P}(\mathfrak{A})$}] (ab) at (+3,2) {};
\node[circle,fill,inner sep=0pt,minimum size=3pt,label=above:{$\operatorname{QC}(\mathfrak{A})\times \mathcal{D}(1,0,0)(\mathfrak{A})$}] (qd100) at (+6,2) {};

\node[circle,fill,inner sep=0pt,minimum size=3pt,label=below:{$\mathcal{D}(t,1,1)(\mathfrak{A})$}] (dt11) at (-6,0) {};
\node[circle,fill,inner sep=0pt,minimum size=3pt,label=below:{$\mathcal{D}(1,-1,1)(\mathfrak{A})$}] (d1m11) at (-3,0) {};
\node[circle,fill,inner sep=0pt,minimum size=3pt,label=below:{$\operatorname{QC}(\mathfrak{A})$}] (q) at (0,0) {};
\node[circle,fill,inner sep=0pt,minimum size=3pt,label=below:{$\mathcal{D}(u,1,0)(\mathfrak{A})$}] (dt10) at (+3,0) {};
\node[circle,fill,inner sep=0pt,minimum size=3pt,label=below:{$\mathcal{D}(1,0,0)(\mathfrak{A})$}] (d100) at (+6,0) {};

\draw [right hook-stealth, shorten <= 2 pt, shorten >= 2pt] (qdt11f) edge (nxq);
\draw [right hook-stealth, shorten <= 2 pt, shorten >= 2pt] (abctf) edge (nxq);
\draw [left  hook-stealth, shorten <= 2 pt, shorten >= 2pt] (nderf) edge (nxq);
\draw [left hook-stealth, shorten <= 2 pt, shorten >= 2pt] (abf) edge (nxq);
\draw [left hook-stealth, shorten <= 2 pt, shorten >= 2pt] (qd100f) edge (nxq);

\draw [right hook-stealth, shorten <= 2 pt, shorten >= 2pt] (dt11) edge (qdt11);
\draw [right hook-stealth, shorten <= 2 pt, shorten >= 2pt] (dt11) edge node [above=10pt,left=-5pt, rotate=35] {$\varphi_{1}$; \scriptsize  \mbox{$t=1-2s$}} (abct);
\draw [left hook-stealth, shorten <= 2 pt, shorten >= 2pt] (d1m11) edge node [left=0, above=-20pt] {$\varphi_{2}$} (abct);
\draw [left hook-stealth, shorten <= 2 pt, shorten >= 2pt] (q) edge (qdt11);
\draw [right hook-stealth, shorten <= 2 pt, shorten >= 2pt] (q) edge (qd100);
\draw [left hook-stealth, shorten <= 2 pt, shorten >= 2pt] (dt10) edge node [above=-2pt,left=-20pt, rotate=-20] {$\varphi_{3}$}  (abct);
\draw [left hook-stealth, shorten <= 2 pt, shorten >= 2pt] (dt10) edge node [left=0, above=-20pt] {$\varphi_{4}$}(ab);
\draw [right hook-stealth, shorten <= 2 pt, shorten >= 2pt] (dt10) edge node [below] {$\supseteq$} (q);
\draw [left hook-stealth, shorten <= 2 pt, shorten >= 2pt] (d100) edge (ab);
\draw [left hook-stealth, shorten <= 2 pt, shorten >= 2pt] (d100) edge (qd100);
\draw [left hook-stealth, shorten <= 2 pt, shorten >= 2pt] (ab) edge node [above] {$\varphi_{5}$}(nder);
\draw [right hook-stealth, shorten <= 2 pt, shorten >= 2pt] (abct) edge node [above] {$\varphi_{6}$}(nder);
\draw [left hook-stealth, shorten <= 2 pt, shorten >= 2pt] (d1m11) edge node [below] {$\subseteq$} (q);
\end{tikzpicture}

The respective functions are given by

\begin{itemize}
 \begin{multicols}{2}
  \item $\varphi_{1}(D):=((1-2s)D, D, D)$
  \item $\varphi_{2}(D):=(D,-D,D)$
  \item $\varphi_{3}(M):=(uM,(1-u-s)M,(u+s)M)$
  \item $\varphi_{4}(M):=(uM,M)$
  \item $\varphi_{5}(A,B):=(A,\frac{1}{2}B)$
  \item $\varphi_{6}(A,B,C):=(A,\frac{1}{2}(B+C))$
 \end{multicols}
\end{itemize}

The function $\varphi_{5}$ is well defined, because our anti-commutative or commutative condition on $\mathfrak{A}$ guarantees that if $A\mu(X,Y) = \mu(BX,Y)$ for all $X, Y \in V$, then $A\mu(Y,X) = \mu(BY,X)$ and so $2A\mu(X,Y) = \mu(BX,Y)+ \mu(X,BY)$. In the same manner we can see that $\mathcal{D}(2t,1,1)(\mathfrak{A}) \supseteq \mathcal{D}(t,1,0)(\mathfrak{A})$. Also, it is important to notice that $\varphi_{6} \circ \varphi_{3} (M) =  \varphi_{5} \circ \varphi_{4} (M) = (uM, \frac{1}{2}M)$.

We also have the injective functions:
\begin{itemize}
\item If $1 \neq 2(s+u)$, $\varphi_{7}: \mathcal{D}(1-2s,1,1)(\mathfrak{A}) \times \mathcal{D}(u,1,0)(\mathfrak{A}) \rightarrow \mathcal{T}(s)(\mathfrak{A})$, defined by
$\varphi_{7}(D,M):=\varphi_{1}(D) + \varphi_{3}(M)$ .
\item If $t\neq 2u$, $\varphi_{8}: \mathcal{D}(t,1,1)(\mathfrak{A}) \times \mathcal{D}(u,1,0)(\mathfrak{A}) \rightarrow \operatorname{NDer}(\mathfrak{A})$, given by $\varphi_{8}(D,M):=\varphi_{6}\circ \varphi_{1}(D) +  \varphi_{5}\circ \varphi_{4}(M) = (tD,D) + (uM, \frac{1}{2}M).$
\end{itemize}

Under additional assumptions on $\mathfrak{A}$, we can express $\mathcal{T}(t)(\mathfrak{A})$ and $\mathcal{P}(\mathfrak{A})$ in terms of $(\alpha,\beta,\gamma)$-derivations.
Recall that the \textit{centroid} of $\mathfrak{A} = (V,\mu)$, denoted by $\operatorname{C}(\mathfrak{A})$, is the vector space of the $(1,1,0)$-derivations:
$$
\operatorname{C}(\mathfrak{A}):= \{ M \in L^{1}(V;V) : M\mu(X, Y)=\mu(M X,  Y) \mbox{ for all } X,Y \in V  \}.
$$

\begin{theorem}\label{teoremaABCs}
Let  $\mathfrak{A}=(V,\mu)$ be an anti-commutative (or commutative) algebra such that $\operatorname{QC}(\mathfrak{A}) = \operatorname{C}(\mathfrak{A})$. If $t\neq -\frac{1}{2}$, then $\mathcal{T}(t)(\mathfrak{A})$ is a vector space isomorphic to $\mathcal{D}(1-2t,1,1)(\mathfrak{A}) \times \operatorname{C}(\mathfrak{A})$.
\end{theorem}

\begin{proof}
 Consider the linear transformation $\psi: \mathcal{T}(t)(\mathfrak{A}) \rightarrow \mathcal{D}(1-2t,1,1)(\mathfrak{A}) \times \operatorname{C}(\mathfrak{A})$ defined by
$$
\psi(A,B,C) = \left( \frac{1}{2t+1}((t+1)B+ tC), \frac{1}{2t+1}(C-B) \right).
$$
We must show that $\psi$ is well defined. Let $(A,B,C) \in \mathcal{T}(t)(\mathfrak{A})$:
\begin{eqnarray*}
  \left\{
  \begin{array}{l}
    A + tB+(t-1)C=0,\\
    A\mu(X,Y)=\mu(BX,Y) + \mu(X,CY).
  \end{array}
  \right.
\end{eqnarray*}
By Proposition \ref{descomposicion} we have $C - B \in \operatorname{QC}(\mathfrak{A}) = \operatorname{C}(\mathfrak{A})$.

Now, note that
\begin{eqnarray*}
\frac{1-2t}{2t+1}( (t+1)B + tC ) + \frac{1}{2t+1}(C-B) &=& \frac{1}{2t+1}  {\Big[}  ((1-2t)(t+1)-1)B + ((1-2t)t+1)C {\Big]}\\
&=& \frac{1}{2t+1}  {\Big[}  -t(2t+1)B - (t-1)(2t+1)C {\Big]}\\
&=& A.
\end{eqnarray*}
Therefore
\begin{eqnarray*}
\frac{1-2t}{2t+1}( (t+1)B + tC ) \mu(X,Y) & = & \left(A - \frac{1}{2t+1}(C-B) \right)  \mu(X,Y) \\
 & = &  \mu(BX,Y) + \mu(X, CY) - \frac{1}{2t+1}(C-B)\mu(X,Y) \\
 & = &  \mu(BX,Y) + \mu(X, CY) +\frac{1}{2}(C-B)\mu(X,Y) - \frac{1}{2}(C-B)\mu(X,Y)\\
 &    &- \frac{1}{4t+2}(C-B)\mu(X,Y) - \frac{1}{4t+2}(C-B)\mu(X,Y)\\
 & = & \mu\left( [B + \frac{1}{2}(C-B) - \frac{1}{4t+2}(C-B) ]X,Y \right) \\
 &   &+ \mu\left( X, [C - \frac{1}{2}(C-B) - \frac{1}{4t+2}(C-B) ]Y \right) \\
 & = & \mu\left( \frac{1}{2t+1}[(t+1)B + tC] X, Y \right)
 +\mu\left( X , \frac{1}{2t+1}[(t+1)B + tC] Y\right),  \\
\end{eqnarray*}
where the next-to-last line follows by using that $(C-B)$ is in $\operatorname{QC}(\mathfrak{A})$, which is equal to $\operatorname{C}(\mathfrak{A})$ by hypothesis. Therefore, $\frac{1}{2t+1} ((t+1)B + tC)$ is in $\mathcal{D}(1-2t,1,1)$.

Since $\mathcal{D}(1,1,0)(\mathfrak{A}) = \operatorname{C}(\mathfrak{A})$, it is straightforward to verify that the function $\psi$ is the inverse of $\varphi_{7}: \mathcal{D}(1-2s,1,1)(\mathfrak{A}) \times \mathcal{D}(1,1,0)(\mathfrak{A}) \rightarrow \mathcal{T}(s)(\mathfrak{A})$, $\varphi_{7}(D,M) = ((1-2t)D + M,D -tM,D+(t+1)R)$.

\end{proof}

\begin{remark}
If $\mathfrak{A}$ is an anti-commutative (or commutative) algebra such that $\operatorname{QC}(\mathfrak{A}) = \operatorname{C}(\mathfrak{A})$, it is a simple matter to show that
$\mathfrak{A}$ must be centerless or perfect; i.e. $\operatorname{Z}(\mathfrak{A})=\{0\}$ or $\mathfrak{A}=\mu(\mathfrak{A},\mathfrak{A})$.
Although the condition $\operatorname{QC}(\mathfrak{A}) = \operatorname{C}(\mathfrak{A})$ looks restrictive, it is worth pointing out that such equality is satisfied by some families of Lie algebras;  for example, if $\mathfrak{A}$ is a centerless and perfect Lie algebra (see \cite[Subsection 5.4]{legerluks}).

Other remarkable condition studied in \cite{legerluks} is the equality between the quasiderivations of $\mathfrak{A}$, $\operatorname{QDer}(\mathfrak{A})$, and the vector space $\operatorname{Der}(\mathfrak{A})+  \operatorname{C}(\mathfrak{A})$ (see for instance \cite[Corollary 4.16]{legerluks}). Note that if $\mathfrak{A}$ is an anti-commutative (or commutative) algebra which is centerless or perfect, then $\operatorname{Der}(\mathfrak{A}) \cap \operatorname{C}(\mathfrak{A}) = \{0\}$.
\end{remark}

\begin{proposition}
 Let  $\mathfrak{A}=(V,\mu)$ be a perfect anti-commutative (or commutative) algebra such that $\operatorname{Der}(\mathfrak{A}) \oplus  \operatorname{C}(\mathfrak{A}) = \operatorname{QDer}(\mathfrak{A})$. Then if $t\neq1,2$, we have $\mathcal{D}(t,1,1)(\mathfrak{A})= \{0\}$, and $\mathcal{D}(2,1,1)(\mathfrak{A}) =  \operatorname{C}(\mathfrak{A})$.
\end{proposition}

\begin{proof}
Since $ \mathcal{D}(t,1,1)(\mathfrak{A}) \subseteq \operatorname{QDer}(\mathfrak{A})$, if $Q \in \mathcal{D}(t,1,1)(\mathfrak{A})$, then there exist a $D \in \operatorname{Der}(\mathfrak{A})$ and $M \in \operatorname{C}(\mathfrak{A})$, such that $Q = D + M$.

By definition of $(t,1,1)$-derivation,
\begin{eqnarray*}
 t(D+M)\mu(X,Y) & = & tQ \mu(X,Y)\\
                & = & \mu(QX,Y) + \mu(X,QY)\\
                & = & \mu((D+M)X,Y) + \mu(X,(D+M)Y)\\
                & = & D\mu(X,Y) + 2M\mu(X,Y)
\end{eqnarray*}
Since $\mathfrak{A}$ is perfect, it follows $(t-1)D = (2-t)M = 0$. Thus, if $t\neq 1,2$, we have $D=0$ and $M=0$, which makes $\mathcal{D}(t,1,1)(\mathfrak{A}) = \{0\}$. In the case $t=2$, we have $D=0$, and so $\mathcal{D}(2,1,1)(\mathfrak{A}) \subseteq \operatorname{C}(\mathfrak{A})$. Since $\operatorname{C}(\mathfrak{A}) = \mathcal{D}(1,1,0)(\mathfrak{A})$ is a subspace of $\mathcal{D}(2,1,1)(\mathfrak{A})$, then $\mathcal{D}(2,1,1)(\mathfrak{A}) = \operatorname{C}(\mathfrak{A})$.
\end{proof}

\begin{corollary}
Under the hypotheses of the preceding proposition, if $u\neq \frac{1}{2},1$, then $D(u,1,0)(\mathfrak{A}) = \{0\}$.
\end{corollary}

The preceding proposition and the following result shows that when $t=-\frac{1}{2}$, the analogue of Theorem \ref{teoremaABCs} do not hold.

\begin{proposition}
Let  $\mathfrak{A}=(V,\mu)$ be a perfect anti-commutative (or commutative) algebra such that $\operatorname{QC}(\mathfrak{A}) = \operatorname{C}(\mathfrak{A})$ and $\operatorname{Der}(\mathfrak{A}) \oplus  \operatorname{C}(\mathfrak{A}) = \operatorname{QDer}(\mathfrak{A})$. Then $\mathcal{T}(-\frac{1}{2})(\mathfrak{A})$  is a vector space isomorphic to $\operatorname{C}(\mathfrak{A})$.
\end{proposition}

\begin{proof}
If $(A,B,C) \in \mathcal{T}(-\frac{1}{2})(\mathfrak{A})$, then then by Proposition \ref{descomposicion}, $B+C \in \operatorname{QDer}(\mathfrak{A})$ and $B-C \in \operatorname{QC}(\mathfrak{A})$. The hypothesis then implies that there are $D \in \operatorname{Der}(\mathfrak{A})$, $M, N \in \operatorname{C}(\mathfrak{A})=\operatorname{QC}(\mathfrak{A})$ such that $B+C = 2(D+M)$ and $B-C=2N$. Thus, $B=D+M+N$ and $C=D+M-N$. By definition of  $\mathcal{T}(-\frac{1}{2})(\mathfrak{A})$, $2 A = B + 3C$ and
\begin{eqnarray*}
  (B+3C)\mu(X,Y) & = & 2A\mu(X,Y) \\
                 & = & \mu(2BX,Y) + \mu(X,2C)\\
                 & = & \mu(2(D+M+N)X,Y) + \mu(X,2(D+M-N)Y)\\
                 & = & 2D\mu(X,Y) + 4M\mu(X,Y)\\
                 & = & (2D+4M)\mu(X,Y)
\end{eqnarray*}
Since $B+3C = 4D+4M-2N$ and $\mathfrak{A}$ is perfect, we have $D=N = 0$ (recall $\operatorname{Der}(\mathfrak{A}) \cap \operatorname{C}(\mathfrak{A}) = \{0\}$).

Finally, $D=N = 0$ implies that $(A,B,C)=(M,\frac{1}{2}M,\frac{1}{2}M)$ with $M \in \operatorname{C}(\mathfrak{A})$. Therefore, the injective linear transformation $\varphi_{3}: \mathcal{D}(1,1,0)(\mathfrak{A}) \rightarrow  \mathcal{T}(-\frac{1}{2})(\mathfrak{A})$, $\varphi_{3}(M) = (M,\frac{1}{2}M,\frac{1}{2}M)$ is surjective.

\end{proof}

\begin{proposition}
Let  $\mathfrak{A}=(V,\mu)$ be a perfect anti-commutative (or commutative) algebra such that $\operatorname{QDer}(\mathfrak{A}) = \operatorname{Der}(\mathfrak{A}) \oplus  \operatorname{C}(\mathfrak{A})$. Then $\operatorname{NDer}(\mathfrak{A})$ is isomorphic to $\operatorname{Der}(\mathfrak{A}) \oplus  \operatorname{C}(\mathfrak{A})$.
\end{proposition}

\begin{proof}
  Let $(P,Q) \in \operatorname{NDer}(\mathfrak{A})$ be arbitrary. Then there is $D \in \operatorname{Der}(\mathfrak{A})$ and $M \in \operatorname{C}(\mathfrak{A})$ such that $Q = D + \frac{1}{2}M$. The definition of $\operatorname{NDer}(\mathfrak{A})$ implies that
\begin{eqnarray*}
  P\mu(X,Y) &=& \mu(QX,Y) + \mu(X,QY)\\
         &=& \mu((D + \frac{1}{2}M)X,Y) + \mu(X,(D + \frac{1}{2}M)Y)\\
         &=& D\mu(X,Y)+M\mu(X,Y)\\
         &=& (D+M)\mu(X,Y).
\end{eqnarray*}
Since $\mathfrak{A}$ is perfect, we have $P=D+M$, and so the injective linear transformation $\varphi_{8}: \mathcal{D}(1,1,1)(\mathfrak{A}) \times \mathcal{D}(1,1,0)(\mathfrak{A}) \rightarrow \operatorname{NDer}(\mathfrak{A})$, $\varphi_{8}(D,M)=(D+M,D+\frac{1}{2}M)$ is surjective.
\end{proof}

\begin{proposition}
If  $\mathfrak{A}=(V,\mu)$ be a perfect anti-commutative (or commutative) algebra such that $\operatorname{QC}(\mathfrak{A}) = \operatorname{C}(\mathfrak{A})$, then $\mathcal{P}(\mathfrak{A})$ is isomorphic to $\operatorname{C}(\mathfrak{A})$.
\end{proposition}

\begin{proof}
 If $(A,B) \in \mathcal{P}(\mathfrak{A})$, we have $B \in \operatorname{QC}(\mathfrak{A})$ and hence $A\mu(X,Y)=\mu(BX,Y)=B\mu(X,Y)$ by our assumption that $\operatorname{QC}(\mathfrak{A}) = \operatorname{C}(\mathfrak{A})$.
 Since $\mathfrak{A}$ is perfect, we have $A=B$ and so $\varphi_{4}: \mathcal{D}(1,1,0)(\mathfrak{A}) \rightarrow   \mathcal{P}(\mathfrak{A})$ given by $\varphi_{4}(B):=(B,B)$ is an isomorphism.
 \end{proof}

\begin{remark}
Note that the notion of extended derivations can be extended in the following way: let $k \in \mathbb{N}$ and let $\Omega$ be any fixed matrix in $\mathrm{M}(3k-1 \times 3k , \mathbb{K})$. An $\Omega$-derivation of an algebra $\mathfrak{A}=(V,\mu)$ is a $k$-tuple of triples $((A_1,B_1,C_1),\ldots,(A_k,B_k,C_k))$ in $((L^{1}(V;V))^{3})^k$ such that for $j=1\ldots k$ and $i=1\ldots 3k-1$
$$
  \left\{
  \begin{array}{l}
  \displaystyle \sum_{m=1}^{k} \left( \Omega[i, 3m-2] A_{m} + \Omega[i, 3m-1] B_{m} + \Omega[i, 3m] C_{m} \right) =0 \\
    A_j\mu(X,Y) = \mu(B_j X,Y) + \mu(X,C_jY) \quad \mbox{ for all } X,Y\in V.
  \end{array}
  \right.
$$
Another possible natural extension is to consider linear equations that involve multilineal maps in different $L^{k}(V;V)$-spaces and \textit{multilinear polynomials} on the algebra, and  relations of linear dependence between the multilinear maps of the same degree; such as \textit{twisted cocycles} defined in \cite[\S 2]{NovotnyHrivnak2} or \textit{Leibniz-derivations} introduced in \cite[Definition 2.1]{moens}.
\end{remark}

\section{Applications}\label{applicaciones}

The spaces of extended derivations are \textit{invariant} under isomorphism of algebras, and are easily calculated; because such spaces correspond to a homogeneous system of linear equations which depend on the algebras. More precisely, let $\operatorname{GL}(V)$ denote the \textit{General Linear Group} of the vector space $V$, which acts on $L^{k}(V;V)$ by \textit{change of basis}: given $g \in \operatorname{GL}(V)$, it acts on a $k$-multilinear map $\mu \in L^{k}(V;V)$ by
$$
g\cdot\mu(\square, \square, \ldots, \square) := g \mu (g^{-1}\square,g^{-1}\square, \ldots, g^{-1}\square).
$$
Two algebras $\mathfrak{A}=(V,\mu)$ and $\mathfrak{B}=(V,\lambda)$ are isomorphic if and only if there exists a linear transformation $g \in \operatorname{GL}(V)$ such that $\lambda = g\cdot \mu$.

Let $\Omega=\left(\begin{array}{ccc} a_{1} & a_{2}& a_{3} \\ a_{4} & a_{5} & a_{6} \\ \end{array} \right)$ be any fixed matrix in $\mathrm{M}(2\times 3 , \mathbb{K})$, and consider the function $\chi_{\Omega}: L^{2}(V;V) \rightarrow L^{1}( (L^{1}(V;V))^3;(L^{1}(V;V))^2 \times L^{2}(V;V))$ defined by
$$
\chi_{\Omega}(\mu)(A,B,C)=  (a_1 A + a_2 B + a_3 C, a_4 A + a_5 B + a_6 C,   A\mu(X,Y)-\mu(BX,Y)-\mu(X,CY)).
$$
The function $\chi_{\Omega}$ is $\operatorname{GL}(V)$-equivariant; i.e. $\chi_{\Omega}(g\cdot \mu) = g \star \chi_{\Omega}(\mu)$, where $\star$ denotes the natural action of $\operatorname{GL}(V)$ on the set of maps $T$ from $L^{1}(V;V))^3$ to $L^{1}(V;V))^2 \times L^{2}(V;V)$: $(g \star T)(\square)= g\cdot(T(g^{-1}\cdot \square ))$. Therefore, if $\mathfrak{B}=(V,g\cdot\mu)$, then
$\operatorname{Der}_{\Omega}(\mathfrak{B}) = \operatorname{Ker}\chi_{\Omega}(g\cdot\mu) = g\cdot\operatorname{Ker}\chi_{\Omega}(\mu) = g\cdot \operatorname{Der}_{\Omega}(\mathfrak{A})$.

In fact, as happens with any \textit{continuous} and $\operatorname{GL}(V)$-equivariant function with domain $L^{2}(V;V)$, the extended derivations can be used to study \textit{degenerations} of algebras. Let us assume that $\mathbb{K}$ is the field of the complex numbers and we equip $L^{k}(\mathbb{C}^{n};\mathbb{C}^{n})$ with its Hausdorff vector space topology.

\begin{definition}
  Let $\mathfrak{A}=(\mathbb{C}^{n},\mu)$ and $\mathfrak{B}=(\mathbb{C}^{n},\lambda)$ be two algebras over $\mathbb{C}$. The algebra $\mathfrak{A}$ is said to \textit{degenerate} to $\mathfrak{B}$, if $\lambda$ is in the closure of the $\operatorname{GL}(\mathbb{C}^{n})$-orbit of $\mu$
\end{definition}

Since $\chi_{D}$ is a continuous function and $\operatorname{GL}(\mathbb{C}^{n})$-equivariant, it follows by the upper semi-continuity of the \textit{nullity function} on linear operators that

\begin{proposition}\label{obstruction}
  If $\mathfrak{A}$ degenerate to $\mathfrak{B}$, then $\operatorname{Dim} \operatorname{Der}_{\Omega}(\mathfrak{A})$ is less than or equal to $\operatorname{Dim} \operatorname{Der}_{\Omega}(\mathfrak{B})$.
\end{proposition}

\begin{example}
Consider the following one-parameter family of $7$-dimensional nilpotent Lie algebras over $\mathbb{C}$
$$
\mathfrak{g}_{I}(\alpha)
:=
\left\{
\begin{array}{l}
[{\it e_1},{\it e_2}]={\it e_3},[{\it e_1},{\it e_3}]={\it e_4},[{\it e_1},{
\it e_4}]={\it e_5},[{\it e_1},{\it e_5}]={\it e_6},[{\it e_1},{\it e_6}]={
\it e_7},\\
{[{\it e_2},{\it e_3}]}={\it e_5},[{\it e_2},{\it e_4}]={\it e_6},[{
\it e_2},{\it e_5}]= \left( 1-\alpha \right) {\it e_7},[{\it e_3},{\it e_4}
]=\alpha\,{\it e_7}.
\end{array}
\right.
$$
For all $\alpha \neq 0$, the dimensions of some spaces of extended derivations of $\mathfrak{g}_{I}(\alpha)$ are:

\begin{itemize}
\item $\operatorname{Dim} \operatorname{QC}(\mathfrak{g}_{I}( \alpha ))=8$
\item $\operatorname{Dim} \mathcal{D}(0,1,0)(\mathfrak{g}_{I}(\alpha ))=7$, $\operatorname{Dim} \mathcal{D}(1,1,0)(\mathfrak{g}_{I}(\alpha ))=3$ and $\operatorname{Dim} \mathcal{D}(t,1,0)(\mathfrak{g}_{I}(\alpha ))=2$  otherwise.
\item $\operatorname{Dim} \mathcal{T}(0)(\mathfrak{g}_{I}( \alpha ))= 18$,  $\operatorname{Dim} \mathcal{D}(1,1,1)(\mathfrak{g}_{I}(\alpha ))= 10$,
\item $\operatorname{Dim} \mathcal{T}(\frac{1}{2})(\mathfrak{g}_{I}( \alpha ))= 16$, $\operatorname{Dim} \mathcal{D}(0,1,1)(\mathfrak{g}_{I}(\alpha ))= 10$,
\item  $\operatorname{Dim} \mathcal{T}(\beta)(\mathfrak{g}_{I}( \alpha ) )= 17$ with $\beta$ such that  $2\,\alpha\,{\beta}^{2}+ \left( -5\,\alpha+1 \right) \beta+2\,\alpha =0$,
\item $\operatorname{Dim} \mathcal{T}(t)(\mathfrak{g}_{I}(\alpha))= 16$ if $t \neq 0, \beta$.

\end{itemize}

When $\alpha=1$, we have

\begin{itemize}
\item  $\operatorname{Dim} \mathcal{D}(2,1,1)(\mathfrak{g}_{I}(1))=9$,\newline $\operatorname{Dim} \mathcal{D}(-1,1,1)(\mathfrak{g}_{I}(1))=9$, and \newline if $t\neq 0,\pm \frac{1}{2}, 1$, then $\operatorname{Dim} \mathcal{D}(1-2t,1,1)(\mathfrak{g}_{I}(1))=8$.
\end{itemize}

If $\alpha = \frac{1}{10}$, we have

\begin{itemize}
 \item
  $\operatorname{Dim} \mathcal{D}(2,1,1)(\mathfrak{g}_{I}(\frac{1}{10}))= 10$,
   $\operatorname{Dim} \mathcal{D}(5,1,1)(\mathfrak{g}_{I}(\frac{1}{10}))= 9$, and \newline if
   $t\neq 0, \pm\frac{1}{2}, -2$, then $\operatorname{Dim} \mathcal{D}(1-2t,1,1)(\mathfrak{g}_{I}(\frac{1}{10}))= 8$.
\end{itemize}

And if $\alpha \neq 1, \frac{1}{10}, 0$, we have

\begin{itemize}

 \item

  $\operatorname{Dim} \mathcal{D}(2,1,1)(\mathfrak{g}_{I}(\alpha ))= 9$,
  \newline
 $\operatorname{Dim} \mathcal{D}(1-2\beta,1,1)(\mathfrak{g}_{I}(\alpha ))= 9$ with $\beta$ such that $2\,\alpha\,{\beta}^{2}+ \left( -5\,\alpha+1 \right) \beta+2\,\alpha = 0$ and \newline
if $t\neq 0, \pm\frac{1}{2}, \beta$, then $\operatorname{Dim} \mathcal{D}(1-2t,1,1)(\mathfrak{g}_{I}(\alpha ))= 8$.
\end{itemize}

Now consider the $7$-dimensional nilpotent Lie algebra $\mathfrak{g}_{G}=\mathfrak{g}_{I}(0)$ given by
$$
\mathfrak{g}_{G}:=
\left\{
\begin{array}{l}
[{\it e_1},{\it e_2}]={\it e_3},[{\it e_1},{\it e_3}]={\it e_4},[{\it e_1},{
\it e_4}]={\it e_5},[{\it e_1},{\it e_5}]={\it e_6},
[{\it e_1},{\it e_6}]={\it e_7}, \\
{[{\it e_2},{\it e_3}]}={\it e_5},[{\it e_2},{\it e_4}]={\it e_6},[{
\it e_2},{\it e_5}]={\it e_7}.
\end{array}
\right.
$$
We obtain
\begin{itemize}
\item $\operatorname{Dim} \operatorname{QC}(\mathfrak{g}_{G})= 8$
\item $\operatorname{Dim} \mathcal{D}(0,1,0)( \mathfrak{g}_{G} )=7$, $\operatorname{Dim} \mathcal{D}(1,1,0)( \mathfrak{g}_{G} )=3$ and $\operatorname{Dim} \mathcal{D}(t,1,0)( \mathfrak{g}_{G} )=2$  otherwise.
\item $\operatorname{Dim} \mathcal{T}(0)( \mathfrak{g}_{G} )= 19$ and $\operatorname{Dim} \mathcal{T}(t)( \mathfrak{g}_{G} )= 16$ otherwise.
\item $\operatorname{Dim} \mathcal{D}(1,1,1)(  \mathfrak{g}_{G} )= 11$,
      $\operatorname{Dim} \mathcal{D}(0,1,1)(  \mathfrak{g}_{G} )= 11$,
      $\operatorname{Dim} \mathcal{D}(2,1,1)(  \mathfrak{g}_{G} )= 9$ and \newline
      $\operatorname{Dim} \mathcal{D}(1-2t,1,1)(  \mathfrak{g}_{G} )= 8$ otherwise.
\end{itemize}

It is shown in \cite[p. 1270]{Burde} that $\mathfrak{g}_{I}(\alpha)$ with $\alpha \neq 0$ cannot degenerate to $\mathfrak{g}_{G}$. In the special case $\alpha=1$, $\mathfrak{g}_{I}(1)$ admits a $1$-codimension ideal which is $2$-step nilpotent while $\mathfrak{g}_{G}$ do not admit such an ideal. When $\alpha \neq 0,1$, the proof of $\mathfrak{g}_{I}(\alpha)$ does not degenerate to $\mathfrak{g}_{G}$ follows from a careful analysis of a Borel-orbit of $\mathfrak{g}_{I}(\alpha)$ (see {\cite[Proposition 1.7.]{GrunewaldOHalloran}}).

We can give an alternative proof of the preceding fact by using the above invariants. Suppose on the contrary that for any $\alpha \neq0$, $\mathfrak{g}_{I}(\alpha)$ degenerates to $\mathfrak{g}_{G}$ and let $\beta$ be such that $2\,\alpha\,{\beta}^{2}+ \left( -5\,\alpha+1 \right) \beta+2\,\alpha =0$. It is important to note that $\beta \neq 0$.
We conclude from Proposition \ref{obstruction} that $17 = \operatorname{Dim} \mathcal{T}(\beta)(\mathfrak{g}_{I}( \alpha ) ) \leq \operatorname{Dim} \mathcal{T}(\beta)( \mathfrak{g}_{G} )$. But, $\beta \neq 0$, and so $\operatorname{Dim} \mathcal{T}(\beta)( \mathfrak{g}_{G} ) = 16$; which is a contradiction. Note that, we can also use $(s,1,1)$-derivations to give another proof.
\end{example}



\end{document}